\documentclass{proc-l}
\usepackage{amssymb,verbatim}
\usepackage[active]{srcltx}
\usepackage{graphicx}

\newtheorem{thm}{Theorem}[section]
\newtheorem{lem}[thm]{Lemma}
\newtheorem{cor}[thm]{Corollary}

\theoremstyle{definition}
\newtheorem{dfn}[thm]{Definition}

\newcommand{\pc}{\mathrm{PC}}

\newcommand{\ce}{\mathrm{COR}}

\newcommand{\per}{\mathrm{Per}}
\newcommand{\ol}{\overline}

\newcommand{\sm}{\setminus}

\newcommand{\bd}{\mathrm{Bd}}

\newcommand{\ch}{\mathrm{Ch}}

\newcommand{\ga}{\gamma}
\newcommand{\si}{\sigma}

\newcommand{\om}{\omega}

\newcommand{\nin}{\not\in}

\newcommand{\C}{\mbox{$\mathbb{C}$}}

\newcommand{\disk}{\mathbb{D}}

\newcommand{\uc}{\mathbb{S}^1}

\newcommand{\lam}{\mathcal{L}}

\newcommand{\N}{\mathbb{N}}

\newcommand{\Sh}{\mathrm{Sh}}

\begin{document}

\date{August 26, 2011; revised December 16, 2011, and then January 5, 2012}
\title[Recurrent and periodic points in dendritic Julia sets]
{Recurrent and periodic points in dendritic Julia sets}

\author[A.~Blokh]{Alexander~Blokh}

\thanks{The author was partially
supported by NSF grant DMS--0901038}

\address[Alexander~Blokh]
{Department of Mathematics\\ University of Alabama at Birmingham\\
Birmingham, AL 35294-1170}

\email[Alexander~Blokh]{ablokh@math.uab.edu}




\begin{abstract}
We relate periodic and recurrent points in dendritic Julia sets.
This generalizes well-known results for interval dynamics.
\end{abstract}

\dedicatory{Dedicated to the occasion of A. N. Sharkovskiy's 75th
birthday}

\maketitle

\section{Introduction and the main results}\label{s:intro}

There are two types of results in continuous interval dynamics.
First, these are the results dealing with periods of periodic
points. The main one here is an amazing fact established by A. N.
Sharkovskiy in the beginning of 1960s in \cite{sha64} and describing
the coexistence among periods of periodic points of an interval map.
To state it we need the following definitions (in what follows we
assume the knowledge of few basic notions such as periodic point,
cycle etc; for the sake of completeness we define all other
notions).

\begin{dfn}[Sharkovskiy ordering]\label{d:sharord}
Define the {\it Sharkovskiy ordering\/} for the set $\N$ of positive
integers united with the symbol $2^\infty$ as follows:

$$3\succ 5\succ\dots\succ 2\cdot3\succ 2\cdot5\succ\dots\succ
2^\infty\succ \dots \succ 4\succ 2\succ 1$$

\noindent Denote by $\Sh(k)$ the set of all positive integers $m$
such that $k\succ m$, together with $k$ (except when $k=2^\infty$ in
which case the symbol $2^\infty$ is not included in $\Sh(k)$). Also,
given a map $f$ we denote by $\per(f)$ the set of the periods of all
cycles of $f$ (by the \emph{period} we mean the \emph{least}
period).
\end{dfn}

Now we are ready to state the celebrated Sharkovskiy Theorem \cite{sha64}.

\begin{thm}\label{t:shar1} If $f:[0,1]\to [0,1]$ is a
continuous map, $m\succ n$ and $m\in\per(f)$, then $n\in\per(f)$.
Therefore there exists $k\in\N\cup\{2^\infty\}$ such that
$\per(f)=\Sh(k)$. Conversely, if $k\in\N\cup\{2^\infty\}$ then there
exists a continuous interval map $f$ such that $\per(f)=\Sh(k)$.
\end{thm}

One can safely say that Theorem~\ref{t:shar1} started
\emph{combinatorial one-dimensional dynamics}. Papers in this field
either seek to specify the coexistence of periods of cycles for
interval maps (e.g., such are papers on the so-called ``rotation
theory for interval maps'', see \cite{blo95a, bm97}), or attempt to
extend a version of the result onto other one-dimensional maps, such
as maps of ``graphs'', i.e. of \emph{compact one-dimensional
branched manifolds} (see \cite{alm00} where the main topics in
one-dimensional dynamics are nicely covered and an extensive list of
references is provided).

Results of the second type deal with all limit sets rather than only
periodic orbits. This direction has also been initiated by
Sharkovskiy, who studied maps of the interval from this perspective
as well. Still, these developments seem to be less well-known. To
state one of Sharkovskiy's results in this area (the one which we
will generalize in this paper), we need the following definition.

\begin{dfn}[Limit sets and recurrent points]\label{d:limrec}
Suppose that $g:X\to X$ is a continuous map of a compact metric
space $X$ to itself. Given a point $x\in X$, the sequence of points
$x, g(x),\dots$ is called the \emph{orbit} of $x$. The set $\om(x)$
of all limit points of the orbit of $x$ is said to be the
\emph{limit set of $x$}; a point of a limit set is often called
simply a \emph{limit point}. A point which belongs to its own limit
set is said to be \emph{recurrent}.
\end{dfn}

The next definition is a little less standard.

\begin{dfn}[Center of a dynamical system]\label{d:center}
Suppose that $g:X\to X$ is a continuous map of a compact metric
space $X$. The \emph{center} of a dynamical system $g$ is the
closure of the set of all its recurrent points. Equivalently, the
\emph{center} of $g$ can be defined as the smallest invariant closed
set $C_g$ such that for any invariant probability measure
$\mu$ we have that $\mu(C_g)=1$.
\end{dfn}

The most obvious example of a recurrent point is a periodic point;
in this case the recurrence manifests itself in the most transparent
way. Thus, the center of a map must contain the closure of all
periodic points of the map. On the other hand, the opposite
inclusion fails already in the case of irrational circle rotations.
Thus, it is natural to ask in various cases how the set of all
periodic points is related to the set of recurrent points (and thus
to the center of a dynamical system). More generally, one can ask
how the limit sets of all points are related to the set of periodic
points of a map.

These problems have been considered by A. N. Sharkovskiy in the 1960s
when a variety of results were obtained (see, e.g., \cite{sha64a,
sha66, sha66a, sha67, sha68}); the scope of this paper does not
allow us to go into a detailed description of this series of papers
which, in our view, laid the foundation of the \emph{one-dimensional
topological dynamics}. Rather we concentrate upon the problems,
described in the previous paragraph, and the way they were addressed
in \cite{sha64a} where the following theorem was proven.

\begin{thm}\label{t:shar2}
The center of an interval map coincides with the closure of the set
of all periodic points of the map.
\end{thm}

One way of further developing one-dimensional topological dynamics
was to show that results which are somewhat stronger than
Sharkovskiy's results can be obtained if the maps are taken from a
more narrow class than the class of continuous interval maps. If one
stays within the framework of continuous interval maps  (i.e.,
considers neither discontinuous nor smooth interval maps), then the
most natural such class seems to be that of \emph{piecewise-monotone
continuous} interval maps; by \emph{piecewise-monotone continuous}
interval maps we mean continuous interval maps for which the
interval can be partitioned, by finitely many points, into finitely
many adjacent intervals on each of which the map is (non-strictly)
monotone. The following result is due to Z. Nitecki \cite{nit80}.

\begin{thm}\label{t:nitec}
Suppose that $f$ is a piecewise-monotone continuous interval map.
Then the limit set $\om(x)$ of any point is contained in the closure
of the set of all periodic points of $f$.
\end{thm}

Theorems ~\ref{t:shar2} and ~\ref{t:nitec} are clearly related:
Theorem~\ref{t:nitec} implies Theorem~\ref{t:shar2} in the
piecewise-monotone case, yet Theorem~\ref{t:shar2} holds for all
continuous interval maps. Examples constructed by Sharkovskiy show
that Theorem~\ref{t:nitec} does not hold for all continuous interval
maps. Also, we would like to mention here that in the above
statements of Theorems ~\ref{t:shar2} and ~\ref{t:nitec} we omitted
parts of the original formulations which are not directly related to
the present work.

In this paper we extend Theorems~\ref{t:nitec} and
Theorem~\ref{t:shar2} onto continuous maps of more complicated
topological spaces. As mentioned above, there are quite a few papers
in which dynamics was extended from the interval onto more
complicated but still one-dimensional topological spaces, within
both one-dimensional combinatorial dynamics and one-dimensional
topological dynamics; for the most part this was done for continuous
maps of ``graphs''(see, e.g., \cite{alm00, blo80xx}).

However here we generalize Theorems~\ref{t:shar2} and ~\ref{t:nitec}
onto one-dimensional spaces more complicated than ``graphs''; so far
few results similar to those from one-dimensional dynamics were
obtained for such topological spaces (see, e.g. \cite{mt89, aeo07}).
First we need the following definition.

\begin{dfn}[Dendrites, their points and subarcs]\label{d:dendr}
A \emph{dendrite} is a non-degenerated locally connected continuum which does not
contain Jordan curves. A point $x$ of a dendrite $X$ is called an
\emph{endpoint} of $X$ if $X\sm \{x\}$ is connected, a
\emph{cutpoint} of $X$ if $X\sm \{x\}$ is disconnected and a
\emph{branchpoint} of $X$ if $X\sm \{x\}$ has more than two
components. For any two points $a, b\in X$ there exists a unique
closed arc in $X$ with endpoints $a$ and $b$ denoted $[a, b]$; the
notation $(a, b), (a, b]$ and $[a, b)$ is analogous to similar
notation in the interval case.
\end{dfn}

As dendrites are much more complicated topological spaces than the
interval or even a \emph{tree} (i.e. a dendrite with finitely many
branchpoints), it is natural to adjust some of the definitions for
them so that tools of one-dimensional dynamics will apply.


\begin{dfn}\label{d:arctype}
Let $f:X\to X$ be a continuous self-mapping of a dendrite $X$.
Suppose that for points $x, b\in X$ there exists an arc $[a,
b]\subset X$ and a sequence of images $f^{n_k}(x)\in (a, b)$ of $x$
which converge to $b$.  Then we say that $b$ is a limit point of $x$
\emph{of arc type}; denote the set of all limit points of $x$ of arc
type by $\om^a(x)$. If $b$ is a limit point of $b$ of arc type then
we say that $b$ is a recurrent point \emph{of arc type}.

If $y$ is a limit point of $x$ which is not of arc type, then we
call $y$ a limit point of $x$ \emph{of non-separating type}; denote
the set of all limit points of $x$ of non-separating type by
$\omega^{ns}(x)$. If $y$ is a limit point of $y$ of \emph{
non-separating type}, we call $y$ a recurrent point of
\emph{non-separating} type.
\end{dfn}

By definition, $\om^a(x)\cup \om^{ns}(x)=\om(x)$. Also, observe that
if $b$ is a limit point of $x$  of arc type then infinitely many
points $f^{n_k}(x)$ are cutpoints of $X$. In the case of limit
points of arc type the convergence resembles that which takes place
in the interval case. It is then no wonder that limit points of arc
type and recurrent points of arc type play for dendrites a role
similar to that of limit points and recurrent points on the
interval.

\begin{thm}\label{t:recdend}
Let $f:X\to X$ be a continuous self-mapping of a dendrite $X$. Then
all recurrent points of arc type belong to the closure of the set of
all periodic points.
\end{thm}

Clearly, Theorem~\ref{t:recdend} implies Theorem~\ref{t:shar2} in
the case when $X$ is a finite tree. Indeed, if $X$ is a finite tree then
all its recurrent non-periodic points are of arc type. Hence in that case all recurrent points
belong to the closure of the set of all periodic points. Since the center of $f$
is the closure of all its recurrent points, Theorem~\ref{t:shar2} follows.
Thus, Theorem~\ref{t:recdend} can be viewed as a
generalization of Theorem~\ref{t:shar2} for dendrites. The
corresponding generalization of Theorem~\ref{t:nitec} requires
considering a more narrow class of maps of dendrites; on the other
hand, the results in that case are more precise as we now consider
periodic \emph{cutpoints} rather than just periodic points. As such
class, we choose \emph{topological polynomials on their dendritic
Julia sets}; thus, our research is triggered not only by the desire
to further study one-dimensional dynamics, but also by the interest
to complex, in particular polynomial dynamics (so that the obtained
results can be considered as a part of both one-dimensional and
complex dynamics).

Indeed, it is well-known that a polynomial $P$ on its locally
connected Julia set can be described using the appropriate
\emph{lamination $\sim$}, i.e. specific equivalence relation $\sim$
on the circle $\uc$ (notice that $\sim$-classes in this case are
always finite). The corresponding quotient space $J_\sim$ of $\uc$
is then called \emph{topological Julia set} while the map
$f_\sim:J_\sim\to J_\sim$ induced by $\si_d=z^d:\uc\to \uc$ (here
$d\ge 2$ is the degree of $P$) is called \emph{topological polynomial}.
Since the original Julia set $J_P$ of $P$ is assumed here to be
locally connected, it follows that $P|_{J_P}$ and $f_\sim|_{J_\sim}$
are topologically conjugate.

Even though there are, of course, locally connected Julia sets which
are not dendrites, results on dendritic case as a rule contain the
most substantial steps of the proofs; then these proofs often can be
extended onto all Julia sets modulo overcoming technical
difficulties. Therefore we believe that studying dendritic Julia
sets is a proper way of developing topological dynamics of
polynomials on their locally connected Julia sets (in fact, later on
we plan to extend our results onto the general case of locally
connected Julia sets). We need the following definition.

\begin{dfn}[Persistent cutpoints]\label{d:percut}
A point $x$ is a \emph{persistent cutpoint} of $J_\sim$ if all its
images are cutpoints of $J_\sim$.
\end{dfn}

This is not restrictive as the only cutpoints which are not
persistent are preimages of (some) critical points of $f_\sim$; in
what follows we talk about limit points of \emph{persistent
cutpoints} rather than all cutpoints. Observe also that in
Theorem~\ref{t:limdend} we talk about the entire limit set of $x$
and not only the set $\om^a(x)$ of all limit points of $x$ of arc
type.

\begin{thm}\label{t:limdend}
Let $f_\sim$ be a topological polynomial with dendritic Julia set
$J_\sim$, let $X\subset J_\sim$ be an invariant dendrite and let
$x\in X$ be a persistent cutpoint of $X$. Then $\om(x)$ is contained
in the closure of the set of all periodic cutpoints of $f_\sim|_X$.
In particular, the limit set of any persistent cutpoint $x$ of
$J_\sim$ is contained in the closure $\ol{\pc_\sim}$ of the set
$\pc_\sim$ of all periodic cutpoints of $f_\sim$, and all recurrent
persistent cutpoints belong to $\ol{\pc_\sim}$.
\end{thm}

It is easy to see that Theorem~\ref{t:limdend} implies
Theorem~\ref{t:recdend} for topological polynomials with dendritic
Julia sets. Indeed, a limit point of arc type (in particular, a
recurrent point of arc type) belongs to the appropriate limit set
which is the limit set of a persistent cutpoint. Then by
Theorem~\ref{t:limdend} a limit (recurrent) point of arc type
belongs to the closure of all periodic cutpoints. This statement is
even stronger then that of Theorem~\ref{t:recdend}.

Theorem~\ref{t:limdend} allows us to make conclusions about
invariant measures of $f_\sim$. Namely, we prove
Corollary~\ref{c:meas}; in it when we say that a probability measure
$\mu$ is \emph{supported} on a set $A$ we mean that $\mu(A)=1$.

\begin{cor}\label{c:meas}
Suppose that $\mu$ is a probability invariant measure of $f_\sim$.
Then it can be represented as the convex sum of two probability
invariant measures $\mu_e$ (supported on the set of all
endpoints of $J_\sim$) and $\mu_c$ (supported on the set of
cutpoints of $J_\sim$ intersected with the closure of the set of all
periodic cutpoints).
\end{cor}

\noindent \textbf{Acknowledgments.} The author would like to thank
L. Oversteegen, R. Ptacek and V. Timorin for useful discussions. He
is also grateful to the referee for useful and thoughtful comments.

\section{Preliminaries}\label{s:prel}

In this section we introduce the notions necessary to obtain the
announced results. We also state some useful lemmas.

\subsection{Laminations}\label{ss:lami}

We begin with laminations introduced by Thurston in \cite{thu85}.
Laminations provide a combinatorial tool which allows us to deal
with polynomial dynamics. We define laminations below, however our
approach is somewhat different from \cite{thu85} (cf.
\cite{blolev02a}).

\begin{dfn}[Laminations as equivalence relations]\label{d:lam}

An equivalence relation $\sim$ on the unit circle $\uc$ is called a
\emph{lamination} if it has the following properties:

\noindent (E1) the graph of $\sim$ is a closed subset in $\uc \times
\uc$;

\noindent (E2) if $t_1\sim t_2\in \uc$ and $t_3\sim t_4\in \uc$, but
$t_2\not \sim t_3$, then the open straight line segments in $\C$
with endpoints $t_1, t_2$ and $t_3, t_4$ are disjoint;

\noindent (E3) each equivalence class of $\sim$ is totally
disconnected.
\end{dfn}

Consider the map $\si_d:\uc\to\uc$ defined by the formula
$\si_d(z)=z^d (d\ge 2)$.

\begin{dfn}[Dynamics and invariant laminations]\label{d:si-inv-lam}
A lamination $\sim$ is called ($\si_d$-){\em invariant} if:

\noindent (D1) $\sim$ is {\em forward invariant}: for a class $g$,
the set $\si_d(g)$ is a class too;

\noindent (D2) $\sim$ is {\em backward invariant}: for a class $g$,
its preimage $\si_d^{-1}(g)=\{x\in \uc: \si_d(x)\in g\}$ splits into
at most $d$ classes;

\noindent (D3) for any $\sim$-class $g$, the map $\si_d: g\to
\si_d(g)$ extends to $\uc$ as an orientation preserving covering map
such that $g$ is the full preimage of $\si_d(g)$ under this covering
map.

\noindent (D4) all $\sim$-classes are finite.

\end{dfn}

Part (D3) of Definition~\ref{d:lam} has an equivalent version. A
{\em (positively oriented) hole $(a, b)$ of a compactum $Q\subset
\uc$} is a component of $\uc\sm Q$ such that moving from $a$ to $b$
inside $(a, b)$ is in the positive direction. Then (D3) is
equivalent to the fact that for a $\sim$-class $g$ either $\si_d(g)$
is a point or for each positively oriented hole $(a, b)$ of $g$ the
positively oriented arc $(\si_d(a), \si_d(b))$ is a positively
oriented hole of $\si_d(g)$.

For a $\si_d$-invariant lamination $\sim$ we consider the
\emph{topological Julia set} $\uc/\sim=J_\sim$ and the
\emph{topological polynomial} $f_\sim:J_\sim\to J_\sim$ induced by
$\si_d$. The quotient map $p_\sim:\uc\to J_\sim$ extends to the
plane with the only non-trivial fibers being the convex hulls of
$\sim$-classes. Using Moore's Theorem one can extend $f_\sim$ to a
branched-covering map $f_\sim:\C\to \C$ of the same degree. The
complement of the unbounded component of $\C\sm J_\sim$ is called
the \emph{filled-in topological Julia set} and is denoted $K_\sim$.
If the lamination $\sim$ is fixed, we may omit $\sim$ from the
notation.

For points $a$, $b\in\uc$, let $\ol{ab}$ be the {\em chord} with
endpoints $a$ and $b$ (if $a=b$, set $\ol{ab}=\{a\}$). For
$A\subset\uc$ let $\ch(A)$ be the \emph{convex hull} of $A$ in $\C$.

\begin{dfn}[Geometric laminations, their leaves and gaps]\label{d:lea}
If $A$ is a $\sim$-class, call an edge $\ol{ab}$ of $\bd(\ch(A))$ a
\emph{leaf}. 
The family of all leaves of $\sim$, denoted by $\lam_\sim$, is
called the \emph{geometric lamination generated by $\sim$}. Denote
the union of all leaves of $\lam_\sim$ by $\lam^+_\sim$. Extend
$\si_d$ (keeping the notation) linearly over all \emph{individual
chords} in $\ol{\disk}$, in particularly over leaves of $\lam_\sim$.
Note, that even though the extended $\si_d$ is not well defined on
the entire disk, it is well defined on every individual chord in the
disk.

The two-point $\sim$-class (and its convex hull) is said to be a
\emph{leaf-class}. The closure of a non-empty component of $\disk\sm
\lam^+_\sim$ is called a \emph{gap} of $\sim$. If $G$ is a gap, we
talk about \emph{edges of $G$}; thus a leaf is either a leaf-class,
or an edge of a gap. If $G$ is a gap or leaf, we call the set
$G'=\uc\cap G$ the \emph{basis of $G$}. A degenerate $\sim$-class is
said to be a \emph{bud} of $\sim$. In what follows for $y\in J_\sim$
we denote by $G_y=\ch(p^{-1}_\sim(y))$ the gap, leaf-class or bud
corresponding to $y$ under $p_\sim$.
\end{dfn}

A gap or leaf $U$ is said to be \emph{preperiodic} if for some minimal
$m$ the set $\si_d^m(U)$ is periodic, $m>0$, and $U, \dots,
\si_d^{m-1}(U)$ are not periodic. Then the number $m$ is called the
\emph{preperiod} of $U$. If $U$ is either periodic or preperiodic,
we will call it \emph{(pre)periodic}. Similarly we treat critical,
precritical and (pre)critical objects.

\subsection{Existence of fixed cutpoints}\label{ss:lamise}

In this subsection we state the results of \cite{bfmot10} concerning
the existence of fixed cutpoints in non-invariant continua (in
particular, non-invariant subcontinua of $J_\sim$). The main results
of \cite{bfmot10} are much more general, however we only need those
of them which apply to topological polynomials with dendritic Julia
sets. We will show how to modify some of the results of
\cite{bfmot10} to our needs. However first we need a few definitions
introduced in \cite{bfmot10}.

\begin{dfn}[Boundary scrambling for dendrites]\label{d:bouscr}
Suppose that $f$ maps a dendrite $D_1$ to a dendrite $D_2\supset
D_1$. Put $E=\ol{D_2\sm D_1}\cap D_1$ (observe that $E$ may be
infinite). If for each \emph{non-fixed} point $e\in E$, $f(e)$ is
contained in a component of $D_2\sm\{e\}$ which intersects $D_1$,
then we say that $f$ has the \emph{boundary scrambling property} or
that it \emph{scrambles the boundary}. Observe that if $D_1$
\emph{is} invariant then $f$ automatically scrambles the boundary.
\end{dfn}

Now we can state a combined and simplified version of Lemma 7.2.2(2)
and Lemma 7.2.5 of \cite{bfmot10}. 

\begin{lem}\label{l:fxpt}
The following facts hold.

\begin{enumerate}

\item Suppose that $f$ maps a dendrite $D_1$ to a dendrite
$D_2\supset D_1$. Put $E=\ol{D_2\sm D_1}\cap D_1$. Moreover, suppose
that $f$ scrambles the boundary. Then $f$ has a fixed point $a\in D_1$.

\item If in the above situation that there are no fixed points
in $E$, $f=f_\sim$ is a topological polynomial, and $D_2\subset
J_\sim$ is a subcontinuum of a dendritic topological Julia set, then
$a$ can be chosen to be a cutpoint of $D_1$.

\end{enumerate}

\end{lem}

Yet another result from \cite{bfmot10} is Lemma 7.2.2(1) which is
stated below. When talking about points in a dendrite $D$, we say
that a point $x$ \emph{separates} a point $y$ from a point $z$ if
$y$ and $z$ belong to distinct components of $D\sm \{x\}$.

\begin{lem}\label{l:fixinarc}
Suppose that $f:X\to X$ is a continuous self-mapping of a dendrite
$X$. Suppose that $a\ne b$ are points in $X$ such that $f(a)$ is separated
from $b$ by $a$ and $f(b)$ is separated from $a$ by $b$. Then there
exists a fixed point in $(a, b)$.
\end{lem}

Finally, we state a result which immediately follows from Theorem
7.2.6 of \cite{bopt11} and well-known properties of periodic points
of topological Julia sets; speaking of \emph{periodic cutpoints} of
a map $g:Y\to Y$ we mean cutpoints of $Y$ which are periodic (thus,
if $Y\subset Z$ then we do not consider cutpoints of $Z$ which are
endpoints of $Y$ as periodic cutpoints of $f:Y\to Y$).

\begin{thm}\label{t:infcutper}
Suppose that $X\subset J_\sim$ is an invariant subdendrite of a
topological Julia set $J_\sim$. Then there are infinitely many
periodic cutpoints of $f_\sim|_X$.
\end{thm}

\subsection{Dynamical core of topological polynomials}\label{ss:dyco}

There are a few new results, which to an extent relate the set of
periodic cutpoints of $f_\sim$ to the set of limit points of
persistent cutpoints as well as limit sets of some critical points.
These results were recently obtained in \cite{bopt11}, Section 3. In
the case when $J_\sim$ is a dendrite the main result (Theorem 3.12)
of Section 3 of \cite{bopt11} can be stated as follows.

\begin{thm}\label{t:dyco}
In the case of a topological polynomial $f_\sim$ with dendritic
Julia set $J_\sim$ the minimal invariant continuum containing limit
sets of all persistent cutpoints of $f_\sim$ and the minimal
invariant continuum containing all periodic cutpoints of $f_\sim$
coincide. Moreover, this continuum (denote it $\ce_{f_\sim}$)
coincides with the smallest invariant continuum containing all critical points
of $f_\sim$ which belong to $\ce_{f_\sim}$.
\end{thm}

The continuum defined in Theorem~\ref{t:dyco} is called the
\emph{dynamical core} of $f_\sim$. Clearly, Theorem~\ref{t:dyco}
relates the sets of points which we want to study. However this
connection is not sufficiently precise as in Theorem~\ref{t:dyco} we
deal with minimal continua containing certain sets of points (such
as the union of all limit points of persistent cutpoints and the set
of all periodic cutpoints) rather than with these sets themselves.
The present paper seeks to improve and specify these results by
establishing, at least in the case of dendrites, the connection
between the sets themselves.

We will need the following lemma which is a simplified version of
Lemma 3.11 of \cite{bopt11} as applies in the case when $J_\sim$ is
a dendrite.

\begin{lem}\label{l:dyco}
Suppose that $X\subset J_\sim$ is an invariant continuum and $x\in
X$ is a cutpoint of $X$. Then there exists $n$ such that
$f^n_\sim(x)$ belongs to the minimal invariant continuum containing all
critical points of $f_\sim$ which belong to $X$.
\end{lem}

\section{Main results}\label{s:mt}

For brevity in what follows we will often omit $\sim$ from the
notation (thus, we write $J$ instead of $J_\sim$, $f$ instead of
$f_\sim$, etc.). Also, we often write $\si$ instead of $\si_d$.

We begin by considering the case of an endpoint. It turns out to be
easier, still it shows the way our tools apply.

\begin{lem}\label{l:endpt}
Let $f:X\to X$ be a continuous self-mapping of a dendrite $X$ and
let $b$ be an endpoint of $X$. Suppose that $b$ is a limit point of
$x$ of arc type so that there exists an arc $[a, b]\subset X$ and a
sequence of images $f^{n_k}(x)\in (a, b)$ of $x$ which converge to
$b$. Then $b$ is a limit point of periodic points of $f$. If
$X\subset J$ is an invariant subcontinuum of a topological dendritic
Julia set $J$ and $f$ is a topological polynomial then $b$ is in
fact a limit point of periodic cutpoints of $f|_X$.
\end{lem}

\begin{proof}
Fix some $k$. Let $V$ be the component of $X\sm \{f^{n_k}(x)\}$
which contains $b$. We apply Lemma~\ref{l:fxpt} to $\ol{V}$ and to
the map $f^{n_{k+1}-n_k}$. Then by Lemma~\ref{l:fxpt}(1) $\ol{V}$
contains a periodic point $z$ (actually, a $f^{n_{k+1}-n_k}$-fixed
point $z$); moreover, since $f^{n_k}(x)$ is not
$f^{n_{k+1}-n_k}$-fixed, then $z\in V$.  Since $X$ is locally
connected, this implies the lemma in the general case. On the other
hand, if $f$ is a topological polynomial with dendritic Julia set
$J$ and $X\subset J$ then by Lemma~\ref{l:fxpt}(2) $b$ is a limit
point of periodic cutpoints of $f$ as desired.
\end{proof}

Clearly, this lemma proves Theorem~\ref{t:recdend} in the case when
$b$ is an endpoint of $X$. It also proves in part Theorem~\ref{t:limdend} by
showing that, in the case of a topological dendritic Julia set $J$
and a topological polynomial an endpoint $b$ of $X\subset J$ which
is a limit point of arc type is a limit point of periodic cutpoints.
To deal with the general case we need the
following result; it has a technical nature but implies a lot of
useful conclusions.

\begin{lem}\label{l:onesidelim}
Let $f:X\to X$ be a continuous self-mapping of a dendrite $X$.
Suppose that $b$ is a limit point of $x$ of arc type so that there
exists an arc $[a, b]\subset X$ and a sequence of images
$f^{n_k}(x)\in (a, b)$ of $x$ which converge to $b$. Moreover,
suppose that there exists $d\in (a, b)$ such that the component $B$
of $X\sm\{d, b\}$, containing $(d, b)$, has the following
properties:

\begin{enumerate}

\item if we make no extra assumptions about $f$ and $X$,
then we assume that $B$ does not contain any \emph{periodic points};

\item if we are given that $X\subset J_\sim$ is a subset of a
dendritic Julia set and $f_\sim$ is a topological polynomial, then
we assume only that $B$ does not contain \emph{periodic cutpoints of $f|_X$}.

\end{enumerate}

Then $b$ never enters $(d, b)$ and for every point $y\in (d, b)$ and
any number $m$ such that $f^m(y)\in (d, b)$ we have that $f^m(y)\in
(y, b)$. In particular, we may assume that $x\in (a, b)$, $f^{n_1}(x), \dots,
f^{n_k}(x), \dots$ are all images of $x$ which enter $(x, b)$, and
that these points approach $b$ in a monotone fashion.
\end{lem}

\begin{proof}
Let us introduce the following order among points of $[a, b]$: $z<y$
means that $y\in (z, b]$. We may assume that $d<x<f^k(x)$ for some
$k$. Let us show that then $f^{2k}(x)$ is contained in the component
of $X\sm \{f^k(x)\}$ which contains $b$. Indeed, suppose otherwise. Then
there are two cases. First, it may happen that $f^{2k}(x)$ is
located in a component $V$ of $X\sm \{f^k(x)\}$ which contains
neither $x$ nor $b$. Then this component $V$ is inside $B$ and it
follows by Lemma~\ref{l:fxpt}, applied to $V$ and $f^k$, that there
is a periodic (actually, $f^k$-fixed) point (in the case (2) of the
lemma, cutpoint) inside $B$, a contradiction. Second, $f^k(x)$ may be
contained in the component of $X\sm \{f^k(x)\}$ containing $x$. Then as
$V$ we can consider the component of $X\sm \{x, f^k(x)\}$ containing
$(x, f^k(x))$. Again by Lemma~\ref{l:fxpt} this implies that there
is a periodic (actually, $f^k$-fixed) point (in the case (2) of the
lemma, cutpoint) inside $B$, a contradiction. Hence $f^{2k}(x)$ belongs
to the component of $X\sm \{f^k(x)\}$ which contains $b$.

The arguments can be continued by induction. Indeed, assume that
$f^{nk}(x)$ is contained in the component of $X\sm \{f^k(x)\}$ which
contains $b$. Consider $f^{(n+1)k}(x)=f^{nk}(f^k(x))$. As before,
there are three possible types of locations of $f^{(n+1)k}(x)$. If
$f^{(n+1)k}(x)$ is located in a component $V$ of $X\sm \{f^k(x)\}$
which contains neither $x$ nor $b$, then $V$ is inside $B$ and it
follows by Lemma~\ref{l:fxpt} applied to $V$ and $f^{nk}$ that there
is a periodic (actually, $f^{nk}$-fixed) point (in the case (2) of
the lemma cutpoint) inside $B$, a contradiction.

On the other hand, if $f^{(n+1)k}(x)$ is contained in the component
of $X\sm \{f^k(x)\}$ containing $x$, then as $V$ we can consider the
component of $X\sm \{x, f^k(x)\}$ containing $(x, f^k(x))$. Again by
Lemma~\ref{l:fxpt} applied to $V$ and $f^{nk}$ this implies that
there is a periodic (actually, $f^{nk}$-fixed) point (in the case
(2) of the lemma, cutpoint) inside $B$, a contradiction. Hence
$f^{(n+1)k}(x)$ is contained in the component of $X\sm \{f^k(x)\}$
which contains $b$. By induction we see that for all integers $n\ge
1$ we have that $f^{nk}(x)$ is contained in the component of $X\sm
\{f^k(x)\}$ which contains $b$.

Consider the point $b$. Suppose that $f^t(b)\in (d, b)$. Now the
arguments from the previous paragraph show that for any $r>1$ the
point $f^{rt}(b)$ belongs to the component of $X\sm \{f^t(b)\}$
containing $d$. Consider $f^{kt}$-images of $b$ and $x$; it follows
that $f^{kt}(b)$ is in the component of $X\sm \{b\}$ containing $d$
while $f^{kt}(x)$ is in the component of $X\sm \{x\}$ containing
$b$. By Lemma~\ref{l:fxpt} this implies that there exists a periodic
point in $B$ (in the case (2) of the lemma, cutpoint), a contradiction.

Consider now a point $y\in (d, b)$. Suppose that for some $t$ we
have that $f^t(y)\in (d, b)$. Then $f^t(y)$ belongs either to the
component of $X\sm \{y\}$ containing $d$ or to the component of
$X\sm \{y\}$ containing $b$. Consider first the case when $f^t(y)$
belongs to the component of $X\sm \{y\}$ containing $d$. Then the
component $V$ of $X\sm \{y\}$ which contains $f^t(y)$ must contain
$d$. Now the arguments from the previous paragraph show that for any
$r>1$ the point $f^{rt}(y)$ belongs to the component of $X\sm
\{f^t(y)\}$ containing $d$. Consider $f^{kt}$-images of $y$ and $x$;
it follows that $f^{kt}(y)$ is in the component of $X\sm \{y\}$
containing $d$ while $f^{kt}(x)$ is in the component of $X\sm \{x\}$
containing $b$.

Now the argument depends on the mutual location of $y$ and $x$.
Suppose that $y$ separates $x$ from $d$ (and so the order of points
is $d<y<x<b$). Then we are exactly in the situation of
Lemma~\ref{l:fixinarc} as applies to the arc $[y, x]$ and the map
$f^{kt}$. Hence there exists a $f^{kt}$-fixed point in $(y, x)$, a
contradiction with the properties of $d$. Now assume that $x$
separates $y$ from $d$ (and so the order of points is $d<x<y<b$).
Then Lemma~\ref{l:fxpt} applies to the component $V$ of $X\sm \{x,
y\}$ containing $(x, y)$ and to the map $f^{kt}$ and shows that
there is a periodic (actually, $f^{kt}$-fixed) point (in the case
(2) of the lemma, cutpoint) inside $B$, a contradiction. Since we
have considered all possible cases and they all lead to a
contradiction, we see that the case when $f^t(y)$ belongs to the
component of $X\sm \{y\}$ containing $d$ is impossible. Hence the
order of points must be $d<y<f^t(y)$. The last claim of the lemma
immediately follows.
\end{proof}

We are ready to prove Theorem~\ref{t:recdend}. Indeed, suppose that
$b$ is a recurrent point of arc type which is not a limit point of
periodic points. We may assume that there exists a point $d$ and a
sequence $\{n_k\}$ such that points $f^{n_k}(b)$ belong to $(d, b)$
and converge to $b$ while the component $B$ of $X\sm \{d, b\}$
containing $(d, b)$ contains no periodic points. Then by
Lemma~\ref{l:onesidelim} we immediately get a contradiction as by
this lemma the point $b$ cannot be mapped to $(d, b)$ at all.

To prove Theorem~\ref{t:limdend} we need to work more, in particular
we need to take into account the fact that topological polynomials
have finitely many critical points whose behavior greatly influences
the dynamics of the map (as an example of such influence one can
consider Lemma~\ref{l:dyco} which plays a useful role in what
follows; in fact, arguments in the proof of the following lemma rely
upon Lemma~\ref{l:onesidelim} and Lemma~\ref{l:dyco}). Recall that
for brevity we write $f, J, \pc$ instead of $f_\sim, J_\sim,
\pc_\sim$ respectively.

\begin{lem}\label{l:onesidelim1}
Let $f$ be a topological polynomial with dendritic Julia set $J$ and
$X\subset J$ be an invariant dendrite. Suppose that $b$ is a limit
point of $x\in X$ of arc type. Then $b$ belongs to the closure
$\ol{\pc(X)}$ of the set $\pc(X)$ of all periodic cutpoints of
$f|_X$.
\end{lem}

\begin{proof}
By way of contradiction we may assume that $b\nin \ol{\pc(X)}$.
Since $b$ is a limit point of arc type, there exists an arc $[a,
b]\subset X$ and a sequence of images $f^{n_k}(x)\in (a, b)$ of $x$
which converge to $b$. On the other hand, the assumption that $b\nin
\ol{\pc(X)}$ implies that there exists $d\in (a, b)$ such that the
component $B$ of $X\sm\{d, b\}$, containing $(d, b)$, does not
contain periodic cutpoints. By Lemma~\ref{l:onesidelim} this implies
that $b$ never enters $(d, b)$ and for every point $y\in (d, b)$ and
any number $m$ such that $f^m(y)\in (d, b)$ we have that $f^m(y)\in
(y, b)$. As before, we introduce the following order among points of
$[a, b]$: $z<y$ means that $y\in (z, b]$. Then by
Lemma~\ref{l:onesidelim} we may assume that $x, f^{n_1}(x), \dots,
f^{n_k}(x), \dots$ are \emph{all} images of $x$ which enter $(d, b)$
and that in fact $x<f^{n_1}(x)<f^{n_2}(x)<\dots, f^{n_i}(x)\to b$.

Consider critical points of $f$. Some of them \emph{never enter}
$B$. Let $c_1, \dots, c_w$ be all critical points of $f$ which
\emph{do enter} $B$ for the first time under the powers $r_1, \dots,
r_w$ of $f$, respectively. Observe, that since we can choose $d$
arbitrarily close to $b$, we may assume that $r_i>0$ for each $i$.
However the proof is valid also if some $r_i$ equal zero. Notice
also, that the points $f^{r_i}(c_i)$ do not have to belong to $(d,
b)$. It is easy to see that then there exists a point $z\in (d, b)$
such that the following holds:

\begin{enumerate}

\item all points $f^{r_i}(c_i), i=1, \dots, w$ belong to the
component of $X\sm \{z\}$ containing $d$, and

\item for every $i=1, \dots, w$ the arc $[f^{r_i}(c_i), z]$
intersects the arc $[d, z]$ over a non-degenerate arc $[v_i, z]$.

\end{enumerate}

Choose a big number $\ga$ such that $z<f^{n_\ga}(x)$. Construct a
set $I$ as follows. Set $q=n_{\ga+1}-n_\ga$ and consider the union
$I$ of all images of $T=[f^{n_\ga}(x), f^{n_{\ga+1}}(x)]\subset [x,
b]$ under $f^q$. Since $f^q(T)$ is non-disjoint from $T$, the set
$I$ is connected. Clearly, $I$ and its closure map to themselves
under $f^q$. The same holds for $f$-images of $I$ and $f$-images of
its closure $\ol{I}$. Let us show that all critical points $c_i, i=1, \dots,
w$ are disjoint from the set $\ol{I}$.

First we claim that images of $T$ never enter the union $Z$ of
components of $X\sm\{d, f^{n_\ga}(x)\}$ which accumulate upon
$f^{n_\ga}(x)$ but are disjoint from $(f^{n_\ga}(x), b)$ (observe
that by the choice of $z$ and $\ga$ all points $f^{r_i}(c_i)$ belong
to one of such components, namely to the component which contains
$z$). Indeed, otherwise let $t$ be the least such number that
$f^t(T)$ enters a component $A\subset Z$. The union $Q$ of all
iterated $f^q$-images of $f^t(T)$ then is a connected set which maps
to itself by $f^q$. If $Q$ is contained in $\ol{B}$ then by
Theorem~\ref{t:infcutper} there are infinitely many periodic
cutpoints in $\ol{Q}\subset \ol{B}$ which implies that there are
infinitely many periodic cutpoints in $B$, a contradiction.

On the other hand, by Lemma~\ref{l:onesidelim}, $d\nin Q$. Hence we
see that $\ol{Q}$ must ``get out of $B$ through $b$'', i.e. that $Q$
must contain $b$. As $Q$ is connected and contains points of $A$, we
see that $f^{n_\ga}(x)\in Q$. This means that for some $i$ we have
$f^{n_\ga}(x)\in f^{t+qi}(T)$. Hence there is a point $y\in T$ such
that $f^{t+qi}(y)=f^{n_\ga}(x)$. Since $B$ contains no periodic
cutpoints, $y\ne f^{n_\ga}(x)$. On the other hand, by
Lemma~\ref{l:onesidelim} $y\ne f^{n_\ga}(x)$ is also impossible.
This contradiction shows that $f$-images of $T$ never enter components of $X\sm \{d,
f^{n_\ga}(x)\}$ which accumulate upon $f^{n_\ga}(x)$ but are
disjoint from $(f^{n_\ga}(x), b)$. In particular, $f^{n_\ga}(x)$ is
an endpoint of $\ol{I}$ and, by continuity, all $f$-images of $\ol{I}$
are disjoint from $Z$.

Clearly, the set $Y=\bigcup_{i=0}^{q-1}\ol{I}$ is $f$-invariant.
Consider a component $Y'$ of $Y$ such that $\ol{I}\subset Y'$. Then
by the above the set $Y'$ is disjoint from $Z$ and $f^{n_\ga}(x)$ is
an endpoint of $Y'$. Since $f^q$ maps $\ol{I}$ to itself, then there
exists the smallest $s$ such that $f^s(Y')\subset Y'$ while sets
$Y', f(Y'), \dots, f^{s-1}(Y')$ are pairwise disjoint. Let us show
that all sets $f^i(Y'), 1\le i\le s-1$ are disjoint from $B$.
Indeed, since $f^{n_i}(x)\to b$ while $f^{n_i}(x)\in (d, b)$ then
$b\in Y'$. This implies that if $f^i(Y')$ is non-disjoint from $B$,
then $f^i(Y')\subset \ol{B}$ (recall, that by the preceding
paragraph this component is disjoint from $Z$ while on the other
hand the point $b$ does not belong to $f^i(Y')$ and so $f^i(Y')$
cannot exit $B$ through $b$). Since $f^i(Y')$ maps to itself by
$f^s$, we see by Lemma~\ref{l:fxpt} that there are periodic cutpoints in
$B$, a contradiction. Hence indeed all sets $f^i(Y'), 1\le i\le s-1$
are disjoint from $B$.

Now, since points $f^{r_i}(c_i)$ belong to $Z$ for all $i=1, \dots,
w$ it follows that points $c_i, i=1, \dots, w$ do not belong to $Y$.
Therefore, by the preceding paragraph and by the choice of critical
points $c_i$ we see that the critical points of $f^s$ which belong
to $Y'$ never enter $B$ under iterations of $f$. Hence the minimal
$f^s$-invariant continuum $T$, containing all critical points of
$f^s$ which belong to $Y'$, is disjoint from $B$. On the other hand,
by Lemma~\ref{l:dyco} every cutpoint of $\ol{I}$ eventually maps to
$T$ and then stays in $T$ under iterations of $f^s$ while staying
away from $B$ under other iterations of $f$ (because under other
iterations of $f$ this cutpoint will stay away from $B$ as shown
above). Applying this to the point $f^{n_{\ga+1}}(x)$ we see that
its forward orbit cannot converge to $b$ from within the arc
$[f^{n_{\ga+1}}(x), b]$, a contradiction with our assumptions which
completes the proof.
\end{proof}

Recall, that in Theorem~\ref{t:limdend} we deal with sets of all
limit points, not only limit points of arc type. To study this more
general situation we use laminations. We investigate what
corresponds to limit points of arc type (and non-separating type) in
the language of laminations. Let $x$ be a persistent cutpoint; then
we may assume that $G_x$ and all its images are gaps or leaf-classes
and that the bases of all images of $G_x$ consist of the same finite
number of points. If $b$ is a limit point of arc type of $x$ then it
follows that a sequence of images of $G_x$ converges \emph{onto an
edge} $\ell$ of $G_b$ or onto a bud $G_b$ (which in this case we
assume to be a degenerate $\ell$) so that each next image in this
sequence of images of $G_x$ separates the previous images in this
sequence from $G_b$ (in such cases we say that a sequence of images
of $G_x$ converges onto $\ell$ \emph{from one side}). On the other
hand, in the case of a limit point $b$ of non-separating type any
sequence of images of $G_x$ which converges to $G_b$ converges
either to an \emph{endpoint} of $\ell$ or to a bud $G_b$ (thus,
diameters of these images of $G_x$ converge to zero) and we may
assume that a converging sequence of images of some $G_x$ with separation
properties, constructed in the definition of convergence of arc
type, does not exist for $b$ (and for $G_b$).

Thus, these two types of limits can be distinguished in the disk as
well. By $\om^a(G_x)$ we mean the set of all leaves and buds
approached by images of $G_x$ from one side and by $\om^{ns}(G_x)$
the set of all the other limits of $G_x$. Then $p_\sim$ maps $\om^a(G_x)$ to
$\om^a(x)$ and $\om^{ns}(G_x)$ to $\om^{ns}(x)$.

\begin{lem}\label{l:limend}
The set $\om^a(G_x)$ is $\si_d$-invariant and dense in $\om(G_x)$. The
set $\om^a(x)$ is $f$-invariant and dense in $\om(x)$.
\end{lem}

\begin{proof}
Let $b$ be a limit point of $x$ of arc type. Then, by definition,
there is an edge $\ell$ of $G_b$ onto which a sequence of images of
$G_x$ converges from one side (in the degenerate case when $G_b$ is
a bud we consider it to be equal to $\ell$). Clearly, the same then
can be said about the image $\si_d(G_b)$ of $G_b$. Thus, $\om^a(G_x)$
is $\si_d$-invariant and $\om^a(x)$ is $f$-invariant as desired.

Now, suppose that $\om^a(x)$ is not dense in $\om(x)$. Then it is
easy to see that there exist arcs $I\subset K\subset \uc$ such that
the following holds:

\begin{enumerate}

\item infinitely many images of $G_x$ intersect $I$ while the
diameters of \emph{all} such images of $G_x$ converge to zero so
that from some time on all images of $G_x$ which intersect $I$ have
bases contained in $K$ (i.e., there exists $k$ such that for all
$m\ge k$ with $\si^m_d(G_x)$ non-disjoint from $I$ we will have that
the basis of $\si^m_d(G_x)$ is contained in $K$), and

\item the sets $\om^a(G_x)$ and $K$ are disjoint.

\end{enumerate}

Suppose that $\si_d^n(G_x)$ has a very small diameter and is
non-disjoint from $I$. Since $\si_d$ is expanding, it follows that for
some $t<n$ the set $\si_d^t(G_x)$ is very close to a critical leaf $c$
and that $\si_d^{n-t}(c)\in K$. As there are finitely many critical
leaves, it follows that there is a critical leaf $c$ which is a
one-sided limit of images of $G_x$ (thus, $c\in \om^a(G_x)$) and
which enters $K$. Clearly, this is a contradiction.
\end{proof}

Theorem~\ref{t:limdend} immediately follows from
Lemma~\ref{l:onesidelim1} and Lemma~\ref{l:limend}. Moreover, we are
ready to prove Corollary~\ref{c:meas}. Indeed, suppose that $f:J\to
J$ is a topological polynomial with dendritic Julia set $J$ and
$\mu$ is an $f$-invariant probability measure. Clearly, any
non-periodic point has zero $\mu$-measure. In particular, this holds
for critical points which map to endpoints of $J$. It follows that
$\mu$ can be represented as a convex combination of two invariant
probability measures, $\mu_e$ (supported on the set of all
endpoints of $J_\sim$) and $\mu_c$ (supported on the set of
all cutpoints of $J_\sim$). Now, by Poincar\'e Recurrence Theorem
\cite{poi1890} the set of recurrent persistent cutpoints has full
$\mu_c$-measure. On the other hand, by Theorem~\ref{t:limdend} all
recurrent persistent cutpoints belong to the closure of the set of
periodic cutpoints. This completes the proof of
Corollary~\ref{c:meas}.

\bibliographystyle{amsalpha}

\end{document}